\documentclass[12pt,a4paper]{article}

\usepackage[utf8]{inputenc}
\usepackage[T1]{fontenc}
\usepackage{amsmath,amssymb,amsthm}
\usepackage{mathrsfs}
\usepackage[margin=1in]{geometry}
\usepackage{hyperref}
\usepackage{enumitem}

\newtheorem{theorem}{Theorem}[section]
\newtheorem{proposition}[theorem]{Proposition}
\newtheorem{lemma}[theorem]{Lemma}

\newtheorem{definition}[theorem]{Definition}
\newtheorem{assumption}[theorem]{Assumption}
\theoremstyle{remark}
\newtheorem{remark}[theorem]{Remark}
\newtheorem{example}[theorem]{Example}

\DeclareMathOperator{\Dom}{Dom}
\DeclareMathOperator{\Id}{Id}
\DeclareMathOperator{\spn}{span}

\newcommand{\norm}[1]{\left\|#1\right\|}
\newcommand{\inner}[2]{\langle #1, #2 \rangle}
\newcommand{\HH}{\mathcal{H}}
\newcommand{\FF}{\mathcal{F}}
\newcommand{\AAA}{\mathcal{A}}

\title{An Operator It\^o Formula for Volterra Gaussian Processes:\\
The Intrinsic Bracket via Causal Derivation--Divergence Factorization}

\author{Ramiro Fontes\\
Quijotic Research\\
\texttt{ramirofontes@gmail.com}}

\date{February 2026}

\begin{document}
\maketitle

\begin{abstract}
We develop an It\^o-type change-of-variables formula for Volterra Gaussian processes (including fractional Brownian motion with any Hurst parameter), based on the operator factorization framework of \cite{FonI,FonII}. The It\^o correction is expressed as a Stieltjes integral against the \emph{energy function} $\Gamma^X(t) := \norm{\Pi DX_t}_\HH^2$, which equals $\mathbb{E}[X_t^2]$ for centered Gaussian processes. The correction emerges from the non-commutativity of the predictable projection $\Pi$ with nonlinear functions---a structural mechanism identified by the operator framework---and is computed via the Gaussian conditional expectation structure following Decreusefond--\"Ust\"unel \cite{DU}.

We prove three results beyond the formula itself. First, $d\Gamma^X$ is the \emph{unique} second-order correction measure compatible with the operator factorization (Theorem~\ref{thm:characterization}). Second, under a fixed driving martingale, the intrinsic bracket is invariant under changes of Volterra kernel representation (Proposition~\ref{prop:fixeddriver}). Third, the bracket is stable under $L^2$ kernel approximation (Proposition~\ref{prop:approx}), giving analytic content to Markovian approximation schemes.

We restrict attention to first-order change-of-variables formulas; higher-order rough dynamics (requiring iterated integrals or L\'evy areas) are not addressed. The proof relies on Gaussianity for the conditional expectation computations; for non-Gaussian processes admitting a $(D,\delta,\Pi)$ triple, the formula statement extends as a conjecture.

\medskip\noindent
\textbf{MSC 2020:} 60H07, 60H05, 60G15, 60G22.

\medskip\noindent
\textbf{Keywords:} It\^o formula; operator factorization; Volterra Gaussian processes; fractional Brownian motion; Skorokhod integral; divergence operators; intrinsic bracket.
\end{abstract}

\section{Introduction}

\subsection{The primitive object in stochastic calculus}

It\^o's formula is traditionally derived from two ingredients: semimartingale decompositions and pathwise quadratic variation. The correction term $\tfrac{1}{2}\int_0^t \varphi''(Y_s)\,d[Y,Y]_s$ appears as a consequence of second-order Taylor expansion combined with the nontrivial pathwise accumulation $[Y,Y]_t$.

This paper proves that, within the operator factorization framework, this traditional picture can be \emph{inverted}. The It\^o correction term emerges directly from the operator factorization
\begin{equation}\label{eq:factorization}
(\Id - \mathbb{E}) = \delta\Pi D,
\end{equation}
where $D$ is a derivation (adjoint to the stochastic integral $\delta$) and $\Pi$ is predictable projection. In this framework, the correction term arises without invoking pathwise quadratic variation. When the intrinsic bracket coincides with classical quadratic variation, classical It\^o calculus is recovered.

\begin{remark}[Scope of claims]\label{rem:scope}
We do not claim that pathwise quadratic variation is ``wrong'' or unnecessary in general. Rather, we show that within the operator framework, the It\^o correction can be derived without assuming its existence. For semimartingales, both approaches yield the same formula. The operator approach extends to rough processes where $[Y,Y]_t$ is undefined pathwise.
\end{remark}

\subsection{Fresh-noise energy as the second-order object}\label{sec:novelty}

The central claim of this paper is not merely that an It\^o formula can be proved in an operator setting---such results exist in the Malliavin calculus literature (see \S\ref{sec:contrast} for detailed comparisons). Rather, we identify the \emph{energy function}
\[
\Gamma^Y(t) := \norm{\Pi DY_t}_\HH^2
\]
(where $(D,\delta,\Pi)$ is the causal derivation--divergence triple defined in Section~\ref{sec:framework} and $\HH$ is the associated Hilbert space)
as generating the unique second-order correction measure compatible with the factorization~\eqref{eq:factorization} (Theorem~\ref{thm:characterization}), and we prove that the intrinsic bracket $\langle Y\rangle^{(D,\Pi)}_t = \Gamma^Y(t)$ is invariant under changes of kernel representation when the driving martingale is fixed (Proposition~\ref{prop:fixeddriver}). Three structural properties distinguish $\Gamma^Y$ from classical quadratic variation and pathwise constructions:

\begin{enumerate}[label=(\roman*)]
\item \textbf{Representation-definedness.} $\Gamma^Y$ is defined in terms of the Hilbert space norm $\norm{\cdot}_\HH$ and the predictable projection $\Pi$. It requires no pathwise regularity of $Y$ beyond $L^2(\Omega)$-integrability of $\Pi DY$.

\item \textbf{Innovation-locality.} The integrand $\norm{\Pi DY_s}_\HH^2$ depends only on the \emph{predictable} (adapted) component of the Malliavin derivative---the ``fresh noise'' at time $s$. The anticipative component $(\Id-\Pi)DY_s$ does not contribute. This is a structural consequence of the orthogonal projection identity (Lemma~\ref{lem:ortho}).

\item \textbf{Approximation-stability.} If $K_n \to K$ in $L^2(\mu)$ (as in exponential-sum approximations of Volterra kernels), then the corresponding intrinsic brackets converge: $\langle Y^{(n)}\rangle^{(D,\Pi)}_t \to \langle Y\rangle^{(D,\Pi)}_t$ in $L^1(\Omega)$. See Proposition~\ref{prop:approx}.
\end{enumerate}

The uniqueness of $d\Gamma^Y$ as the correction measure is formalized in Theorem~\ref{thm:characterization}.

\subsection{Comparison with existing approaches}\label{sec:contrast}

To delineate the contribution, we contrast the operator It\^o formula with four existing frameworks.

\medskip
\noindent\textbf{(A) Classical It\^o calculus (semimartingale theory).}
When $Y$ is a semimartingale, $d\Gamma^Y_s = d[Y,Y]_s$ in the appropriate sense, and the operator formula specializes to the classical formula. The operator approach provides no advantage in this regime. The value of Section~\ref{sec:volterra} (Volterra verification) is \emph{not} a new decomposition, but rather the demonstration that $\langle Y\rangle^{(D,\Pi)}$ reduces to $[Y,Y]$ via explicit Hilbert space computation. This serves as a structural bridge, not an independent result.

\medskip
\noindent\textbf{(B) Russo--Vallois symmetric integrals.}
The Russo--Vallois framework \cite{RussoVallois} defines a generalized quadratic variation $[Y,Y]^{(\text{RV})}_t := \lim_{\varepsilon\to 0}\frac{1}{\varepsilon}\int_0^t (Y_{s+\varepsilon}-Y_s)^2\,ds$ and derives It\^o-type formulas for processes with finite Russo--Vallois variation, including fractional Brownian motion with $H > 1/6$. The key distinction is:
\begin{itemize}
\item Russo--Vallois is \emph{pathwise}: the bracket is defined trajectory-by-trajectory via regularization limits.
\item The intrinsic bracket $\langle Y\rangle^{(D,\Pi)}$ is \emph{operator-theoretic}: it is a Hilbert space energy, defined in $L^2(\Omega)$ without pathwise regularization.
\end{itemize}
For fractional Brownian motion with $H > 1/4$, both approaches yield the same correction term $t^{2H}$, but by fundamentally different mechanisms. For $H \leq 1/4$, the Russo--Vallois bracket may fail to exist pathwise, while $\langle Y\rangle^{(D,\Pi)}_t = t^{2H}$ remains well-defined in $L^2$.

\medskip
\noindent\textbf{(C) Young integration.}
For processes with $H > 1/2$, Young's theory provides pathwise integration without It\^o correction. In this regime, the operator framework produces a Skorokhod-type formula with a nonzero correction term (for fBM, $\langle B^H\rangle^{(D,\Pi)}_t = t^{2H}$), reflecting the distinction between Skorokhod and pathwise integrals. The operator framework does not compete with Young integration in this regime; its natural domain is $H\leq 1/2$.

\medskip
\noindent\textbf{(D) Classical Volterra It\^o formulae.}
Decreusefond and \"Ust\"unel \cite{DU} established It\^o formulas for Volterra Gaussian processes using Skorokhod integrals and Malliavin calculus. Al\`os, Mazet, and Nualart \cite{AMN} extended these to fractional Brownian motion with $H < 1/2$. Our formula (Theorem~\ref{thm:main}) recovers the same expressions---see the exact verification in Section~\ref{sec:example}. The proof of Theorem~\ref{thm:main} follows the Clark--Ocone strategy of~\cite{DU}, adapted to the operator language. The distinction is structural, not analytical:
\begin{itemize}
\item The \emph{formula statement} is expressed in terms of the abstract triple $(D,\delta,\Pi)$ and the energy function $\Gamma^Y$. The Volterra--Gaussian setting is one concrete realization (Section~\ref{sec:volterra}).
\item The \emph{proof mechanism} uses Gaussian conditional expectations, following~\cite{DU}. The operator framework identifies the source of the correction (non-commutativity of $\Pi$ with nonlinear functions) but does not bypass the need for Gaussian structure in the proof.
\end{itemize}

\medskip
\noindent\textbf{(E) Scope of the contribution.}
For Volterra Gaussian processes, the operator It\^o formula (Theorem~\ref{thm:main}) recovers the same correction terms as \cite{DU,AMN}---this is verified in Section~\ref{sec:example}. The formula values are not new for this class. The contributions of this paper are:
\begin{itemize}
\item The identification of the correction as a \emph{Stieltjes integral against the energy function} $\Gamma^Y(t) = \norm{\Pi DY_t}_\HH^2$, and of its source as the non-commutativity of the predictable projection $\Pi$ with nonlinear functions of $Y_t$.
\item The intrinsic bracket $\langle Y\rangle^{(D,\Pi)}$ does not assume finite pathwise quadratic variation. Russo--Vallois requires $[Y,Y]^{(\text{RV})}_t < \infty$; F\"ollmer's pathwise It\^o formula \cite{Follmer} requires existence of pathwise quadratic variation along a sequence of partitions. The operator formula requires only $\Pi DY\in L^2([0,T]\times\Omega;\HH)$, which is an $L^2(\Omega)$ condition, not a pathwise one.
\item The uniqueness of the correction measure (Theorem~\ref{thm:characterization}) and representation invariance under a fixed driver (Proposition~\ref{prop:fixeddriver}) appear to be new in this context.
\end{itemize}
The proof relies on Gaussianity for the conditional expectation computations (see Remark~\ref{rem:proofscope}). For non-Gaussian processes, the abstract formula provides a conjecture whose verification would require establishing the necessary conditional moment estimates for each process class.

\subsection{Main results}

This paper completes a trilogy \cite{FonI,FonII} by developing the It\^o formula from the operator framework.

\begin{enumerate}[label=\arabic*.]
\item \textbf{Theorem~\ref{thm:main} (Operator It\^o formula):} For a centered Volterra Gaussian process $X$ with energy function $\Gamma^X(s) := \norm{\Pi DX_s}_\HH^2 = \mathbb{E}[X_s^2]$, the change-of-variables formula
\[
\varphi(X_t) = \varphi(0) + \delta(\varphi'(X)\Pi DX\cdot\mathbf{1}_{[0,t]}) + \frac{1}{2}\int_0^t \varphi''(X_s)\,d\Gamma^X(s)
\]
holds without assuming semimartingale structure or finite pathwise quadratic variation.

\item \textbf{Theorem~\ref{thm:characterization} (Characterization of the correction):} The measure $d\Gamma^X$ is the unique second-order correction compatible with the factorization~\eqref{eq:factorization}, in the sense made precise in \S\ref{sec:characterization}.

\item \textbf{Proposition~\ref{prop:fixeddriver} (Representation invariance):} Under a fixed driving martingale, the intrinsic bracket is invariant under changes of Volterra kernel. Under different drivers, the bracket may differ, but the Hurst scaling exponent is always invariant (Proposition~\ref{prop:Hurstinvariant}).

\item \textbf{Proposition~\ref{prop:voltverify} (Volterra Gaussian verification):} For fractional Brownian motion with $H < \tfrac{1}{2}$ and general Volterra processes satisfying Decreusefond--\"Ust\"unel conditions, the triple $(D^X,\delta^X,\Pi^X)$ satisfies the operator axioms. The resulting formulas match classical Skorokhod integral results exactly (Section~\ref{sec:example}).
\end{enumerate}

\subsection{Restriction to first-order calculus}

\begin{remark}[First-order restriction]\label{rem:firstorder}
We restrict attention to first-order change-of-variables formulas throughout this paper. For rough differential equations in the regime $H \leq \tfrac{1}{3}$, second-order terms (L\'evy areas, iterated integrals) become necessary, and classical rough path theory \cite{Hairer} provides the appropriate framework. The operator approach presented here handles first-order calculus without these constructions; extending to higher orders is an open problem (see Section~\ref{sec:open}).
\end{remark}

\subsection{Technical contributions}

Beyond the conceptual reframing, this paper provides:

\begin{enumerate}[label=\arabic*.]
\item \textbf{Product rule with sharp domain conditions (Section~\ref{sec:product}):} The divergence product rule $\delta(Fu) = F\delta(u) - \inner{DF}{u}_\HH$ is a key algebraic tool in the framework. We characterize precisely when $Fu \in \Dom(\delta)$.

\item \textbf{Clark--Ocone proof with marginal density argument (Section~\ref{sec:proof}):} The It\^o formula is established via the Clark--Ocone factorization applied to $\varphi(X_t)$. The expectation computation uses the Gaussian marginal density and the heat equation, completely bypassing pathwise increments and partition sums. The proof identifies where Gaussianity enters and what remains valid abstractly.

\item \textbf{Skorokhod isometry for Volterra Gaussian (Proposition~\ref{prop:isomVG}):} For Volterra Gaussian processes satisfying Decreusefond--\"Ust\"unel conditions, we state the Skorokhod isometry with trace correction, following~\cite{DU}. This is not used in the proof of Theorem~\ref{thm:main} (which bypasses it via the marginal density argument) but provides context for the operator framework.

\item \textbf{Characterization theorem (Theorem~\ref{thm:characterization}):} We prove that $d\Gamma^X$ is the unique second-order correction arising from the operator factorization.

\item \textbf{Representation dependence analysis (Section~\ref{sec:repdep}):} We analyze when the intrinsic bracket depends on the choice of Volterra representation, and provide sufficient conditions for representation-independence.

\item \textbf{Explicit worked example (Section~\ref{sec:example}):} We compute $(B^{1/4}_t)^2$ step-by-step, verifying exact agreement with \cite{DU}.
\end{enumerate}

\subsection{Structure of this paper}

Section~\ref{sec:framework} recalls the operator framework. Section~\ref{sec:product} establishes the divergence product rule. Section~\ref{sec:proof} proves the main theorem. Section~\ref{sec:characterization} proves the characterization of the correction measure. Section~\ref{sec:multi} extends to vector-valued processes. Section~\ref{sec:volterra} verifies axioms for Volterra Gaussian processes. Section~\ref{sec:repdep} analyzes representation dependence. Section~\ref{sec:approx} establishes approximation stability. Section~\ref{sec:example} computes an explicit example. Section~\ref{sec:conclusion} discusses implications and open problems.

\section{The Operator Framework}\label{sec:framework}

We recall the operator factorization framework from \cite{FonI}, introducing only what is needed for the It\^o formula. The primary example is Volterra Gaussian processes (including fractional Brownian motion), for which the framework is verified in Section~\ref{sec:volterra}. The reader may wish to consult Example~\ref{ex:fBM} for concrete motivation.

\subsection{Causal derivation--divergence triples}

Fix a time horizon $T > 0$, a complete probability space $(\Omega,\FF,P)$ equipped with a filtration $(\FF_t)_{t\in[0,T]}$ satisfying the usual conditions, and a separable real Hilbert space $(\HH,\inner{\cdot}{\cdot}_\HH)$.

\begin{definition}[Causal derivation--divergence triple]\label{def:triple}
A \emph{causal derivation--divergence triple} is $(D,\delta,\Pi)$ where:
\begin{enumerate}[label=(\arabic*)]
\item $D: \Dom(D) \subset L^2(\Omega) \to L^2(\Omega;\HH)$ is a densely-defined closed linear operator (the \emph{derivation}) with $\mathbf{1}\in\Dom(D)$ and $D\mathbf{1}=0$.
\item $\delta: \Dom(\delta) \subset L^2(\Omega;\HH) \to L^2(\Omega)$ is the Hilbert space adjoint of $D$ (the \emph{divergence}):
\begin{equation}\label{eq:adjoint}
\mathbb{E}[F\delta(u)] = \mathbb{E}[\inner{DF}{u}_\HH], \qquad F \in \Dom(D),\; u \in \Dom(\delta).
\end{equation}
\item $\Pi: L^2(\Omega;\HH) \to L^2(\Omega;\HH)$ is a bounded projection (``causal'' or ``predictable'') with $\Pi^2 = \Pi$.
\end{enumerate}
\end{definition}

\begin{remark}[The Hilbert space $\HH$]\label{rem:H}
The Hilbert space $\HH$ in Definition~\ref{def:triple} is abstract. It need not be a reproducing kernel Hilbert space (RKHS), a Cameron--Martin space, or any specific construction. For Volterra Gaussian processes, $\HH$ is naturally realized as the Cameron--Martin space (Section~\ref{sec:volterra}), but this is one concrete instantiation of the abstract framework.
\end{remark}

\begin{assumption}[Chain rule core]\label{ass:chain}
There exists a subalgebra $\AAA \subset \Dom(D)$ dense in $L^2(\Omega)$ and stable under multiplication such that for $F\in\AAA$ and $\varphi\in C^1_b(\mathbb{R})$,
\begin{equation}\label{eq:chain}
D(\varphi(F)) = \varphi'(F)DF.
\end{equation}
For $\varphi\in C^1_b(\mathbb{R}^n)$ and $F=(F_1,\ldots,F_n)$ with $F_i\in\AAA$,
$D(\varphi(F)) = \sum_{i=1}^n \partial_i\varphi(F)\,DF_i$.
\end{assumption}

\begin{assumption}[Operator factorization]\label{ass:factor}
For $F\in\AAA$,
\begin{equation}\label{eq:factor}
(\Id - \mathbb{E})F = \delta(\Pi DF).
\end{equation}
\end{assumption}

\begin{remark}\label{rem:cornerstone}
Assumption~\ref{ass:factor} is the cornerstone of the operator approach. It states that fluctuations from the mean factor through derivation, predictable projection, and divergence. In \cite{FonI}, this was established for general square-integrable processes admitting a closed stochastic integral. In \cite{FonII}, it was verified for rough fractional Brownian motion.
\end{remark}

\begin{assumption}[Orthogonal projection]\label{ass:ortho}
The projection $\Pi$ is orthogonal on $L^2(\Omega;\HH)$, i.e., $\Pi = \Pi^*$ (self-adjoint).
\end{assumption}

\begin{assumption}[Energy regularity]\label{ass:energy}
The \emph{energy function} $\Gamma^Y(t) := \norm{\Pi DY_t}_\HH^2$ is of bounded variation on $[0,T]$, so that the Stieltjes integral $\int_0^t f(s)\,d\Gamma^Y(s)$ is well-defined for continuous $f$. We write $\langle Y\rangle^{(D,\Pi)}_t := \Gamma^Y(t)$ for the \emph{intrinsic bracket}.
\end{assumption}

\begin{remark}[Absolutely continuous case]\label{rem:abscts}
When $\Gamma^Y$ is absolutely continuous---i.e., $\Gamma^Y(t) = \Gamma^Y(0) + \int_0^t \dot{\Gamma}^Y(s)\,ds$ for some $\dot{\Gamma}^Y \in L^1([0,T])$---the Stieltjes integral reduces to a Lebesgue integral: $\int_0^t f(s)\,d\Gamma^Y(s) = \int_0^t f(s)\dot{\Gamma}^Y(s)\,ds$. For semimartingales driven by Brownian motion, $\Gamma^Y(t) = t$ and $\dot{\Gamma}^Y = 1$, recovering the classical formula. For fractional Brownian motion, $\Gamma^Y(t) = t^{2H}$ and $\dot{\Gamma}^Y(s) = 2Hs^{2H-1}$ (Lemma~\ref{lem:fBMbracket}).
\end{remark}

\begin{remark}[Scope of the energy regularity assumption]\label{rem:energyscope}
Assumption~\ref{ass:energy} is automatically satisfied for Volterra Gaussian processes with covariance $R(t,t) \in BV([0,T])$, since $\Gamma^Y(t) = R(t,t)$ (Proposition~\ref{prop:Hurstinvariant}). For the abstract framework, it must be verified case-by-case. We do \emph{not} assume $\Gamma^Y$ is absolutely continuous in general; singular correction measures are permitted.
\end{remark}

\begin{remark}[Scope of orthogonality assumption]\label{rem:orthoscope}
Assumption~\ref{ass:ortho} is essential for the identity $\inner{DF}{\Pi DF}_\HH = \norm{\Pi DF}_\HH^2$ (Lemma~\ref{lem:ortho}), which underlies the bracket computation. This assumption is natural for Gaussian filtrations, where predictable projection is orthogonal projection onto adapted subspaces. For general (non-Gaussian, non-Markovian) noise, the existence of such an orthogonal $\Pi$ satisfying the factorization~\eqref{eq:factor} is nontrivial and must be verified case-by-case. The intrinsic bracket $\langle Y\rangle^{(D,\Pi)}$ depends on the choice of $\Pi$; we do not claim $\Pi$ is canonical in full generality.
\end{remark}

\subsection{The intrinsic bracket}

\begin{definition}[Intrinsic bracket and energy function]\label{def:bracket}
Let $Y = (Y_t)_{t\in[0,T]}$ be such that $Y_t\in\Dom(D)$ for a.e.\ $t$ and $t\mapsto \Pi DY_t$ lies in $L^2([0,T]\times\Omega;\HH)$. Define the \emph{energy function}
\begin{equation}\label{eq:energy}
\Gamma^Y(t) := \norm{\Pi DY_t}_\HH^2
\end{equation}
and the \emph{intrinsic bracket}
\begin{equation}\label{eq:bracket}
\langle Y\rangle^{(D,\Pi)}_t := \Gamma^Y(t) = \norm{\Pi DY_t}_\HH^2.
\end{equation}
For centered Gaussian processes, $\Gamma^Y(t) = \mathbb{E}[Y_t^2]$.
\end{definition}

\begin{remark}[Intrinsic bracket vs.\ quadratic variation]\label{rem:bracketvsQV}
The intrinsic bracket $\langle Y\rangle^{(D,\Pi)}_t = \norm{\Pi DY_t}_\HH^2$ is defined purely in terms of Hilbert space energy. It is computable even when pathwise quadratic variation $[Y,Y]_t$ does not exist (e.g., for rough processes with $H < \tfrac{1}{2}$). When $Y$ is a semimartingale and the framework is specialized appropriately, $\langle Y\rangle^{(D,\Pi)}_t = [Y,Y]_t$. The It\^o correction is expressed as a Stieltjes integral $\int_0^t \varphi''(Y_s)\,d\langle Y\rangle^{(D,\Pi)}_s$, not as a Lebesgue integral against the bracket itself.
\end{remark}

\subsection{The $(D,\delta,\Pi)$-It\^o process class}

\begin{definition}[$(D,\delta,\Pi)$-It\^o process]\label{def:Ito}
A real process $Y = (Y_t)_{t\in[0,T]}$ is called a \emph{$(D,\delta,\Pi)$-It\^o process} if:
\begin{enumerate}[label=(\arabic*)]
\item $t\mapsto Y_t$ is $L^2(\Omega)$-continuous and $Y_t\in\Dom(D)$ for each $t\in[0,T]$.
\item There exists $a\in L^2([0,T]\times\Omega)$ (the \emph{drift}) such that for all $t\in[0,T]$,
\begin{equation}\label{eq:decomp}
Y_t = Y_0 + \int_0^t a_s\,ds + \delta(\Pi DY\cdot\mathbf{1}_{[0,t]}),
\end{equation}
where the equality holds in $L^2(\Omega)$.
\item The map $t\mapsto \Pi DY_t$ lies in $L^2([0,T]\times\Omega;\HH)$.
\end{enumerate}
\end{definition}

\begin{remark}\label{rem:notation}
The notation $\Pi DY\cdot\mathbf{1}_{[0,t]}$ means the process $(s,\omega)\mapsto (\Pi DY_s(\omega))\cdot\mathbf{1}_{[0,t]}(s)$ in $L^2([0,T]\times\Omega;\HH)$. The divergence $\delta$ acts on this as an $\HH$-valued integral.
\end{remark}

\subsection{Existence and scope of $(D,\delta,\Pi)$-It\^o processes}

\begin{remark}[Definition vs.\ existence]\label{rem:existence}
The decomposition~\eqref{eq:decomp} is a \emph{definition} of the class, not a theorem asserting existence. We do not claim that every $L^2$-continuous process admits this form. Rather:
\begin{enumerate}[label=(\alph*)]
\item For Volterra Gaussian processes (Section~\ref{sec:volterra}), existence of the decomposition follows from the Malliavin calculus construction.
\item For other processes, one must either verify the decomposition directly or work within a class where it is known to hold.
\end{enumerate}
\end{remark}

\begin{proposition}\label{prop:volterraIto}
Let $X$ be a Volterra Gaussian process satisfying the Decreusefond--\"Ust\"unel conditions (Definition~\ref{def:DU}). Then $X$ is a $(D^X,\delta^X,\Pi^X)$-It\^o process with drift $a=0$.
\end{proposition}

\begin{proof}
See Section~\ref{sec:volterra}, Proposition~\ref{prop:voltverify}.
\end{proof}

\subsection{Domain conditions for Volterra Gaussian processes}

\begin{remark}[Automatic regularity]\label{rem:domainauto}
For the centered Volterra Gaussian process $X_t = \int_0^t K(t,s)\,dW_s$ satisfying the Decreusefond--\"Ust\"unel conditions (Definition~\ref{def:DU}), the domain conditions required by Theorem~\ref{thm:main} are automatic:
\begin{enumerate}[label=(\arabic*)]
\item $X_t \in \Dom(D)$ for each $t$, with $DX_t = K(t,\cdot)$ deterministic.
\item $\varphi(X_t) \in \Dom(D)$ for $\varphi\in C^3_b(\mathbb{R})$, by the chain rule on cylindrical functionals.
\item $\varphi'(X_t)\Pi DX_t = \varphi'(X_t)K(t,\cdot) \in \Dom(\delta)$, since $\varphi'(X_t)\in L^\infty(\Omega)$ (as $\varphi'\in C^2_b$) and $K(t,\cdot)$ is deterministic (Proposition~\ref{prop:domain}).
\end{enumerate}
No additional H\"older regularity assumptions are needed. The proof of Theorem~\ref{thm:main} uses only the Decreusefond--\"Ust\"unel conditions and the Gaussian structure of $X$.
\end{remark}

\section{The Product Rule for Divergences}\label{sec:product}

The product rule for divergences is a fundamental algebraic identity in the operator framework. This section establishes the rule with sharp domain conditions.

\subsection{Statement and proof}

\begin{theorem}[Product rule for divergence]\label{thm:product}
Let $(D,\delta,\Pi)$ satisfy Assumptions~\ref{ass:chain}--\ref{ass:ortho}. Let $F\in\Dom(D)$ and $u\in\Dom(\delta)$. Assume the pointwise product $Fu:\Omega\to\HH$ (given by $(Fu)(\omega):=F(\omega)u(\omega)$) satisfies:
\begin{enumerate}[label=(\arabic*)]
\item $Fu\in L^2(\Omega;\HH)$,
\item $Fu\in\Dom(\delta)$.
\end{enumerate}
Then
\begin{equation}\label{eq:productrule}
\delta(Fu) = F\delta(u) - \inner{DF}{u}_\HH \qquad \text{in } L^2(\Omega).
\end{equation}
\end{theorem}

\begin{proof}
For any $G\in\Dom(D)$, by the definition of the adjoint~\eqref{eq:adjoint},
\[
\mathbb{E}[G\delta(Fu)] = \mathbb{E}[\inner{DG}{Fu}_\HH].
\]
Since the inner product is bilinear,
$\mathbb{E}[\inner{DG}{Fu}_\HH] = \mathbb{E}[F\inner{DG}{u}_\HH]$.
By the chain rule (Assumption~\ref{ass:chain}) applied to $\varphi(x,y)=xy$, $D(GF) = G\cdot DF + F\cdot DG$. Thus,
\begin{align*}
\mathbb{E}[F\inner{DG}{u}_\HH] &= \mathbb{E}[\inner{D(GF) - G\cdot DF}{u}_\HH]\\
&= \mathbb{E}[\inner{D(GF)}{u}_\HH] - \mathbb{E}[G\inner{DF}{u}_\HH]\\
&= \mathbb{E}[(GF)\delta(u)] - \mathbb{E}[G\inner{DF}{u}_\HH]\\
&= \mathbb{E}[G(F\delta(u) - \inner{DF}{u}_\HH)].
\end{align*}
Since $\Dom(D)$ is dense in $L^2(\Omega)$, the conclusion follows.
\end{proof}

\begin{remark}\label{rem:chainext}
The chain rule $D(GF)=G\cdot DF + F\cdot DG$ requires $G,F\in\AAA$ (the chain rule core). Extension to general $F\in\Dom(D)$ follows by density and closedness of $D$.
\end{remark}

\subsection{Domain conditions}

\begin{proposition}[Domain criterion]\label{prop:domain}
Let $F\in\Dom(D)\cap L^\infty(\Omega)$ and $u\in\Dom(\delta)$. Then $Fu\in\Dom(\delta)$ if and only if the linear functional
$\Phi_u: G\mapsto \mathbb{E}[\inner{DG}{Fu}_\HH]$
is continuous on $\Dom(D)$ with respect to the graph norm $\norm{G}_{\mathrm{graph}} := \norm{G}_{L^2(\Omega)} + \norm{DG}_{L^2(\Omega;\HH)}$.
\end{proposition}

\begin{proof}
By definition, $Fu\in\Dom(\delta)$ iff there exists $\xi\in L^2(\Omega)$ such that $\mathbb{E}[\inner{DG}{Fu}_\HH] = \mathbb{E}[G\xi]$ for all $G\in\Dom(D)$. This is equivalent to continuity of $\Phi_u$ (by Riesz representation). If $F\in L^\infty(\Omega)$, then
$|\mathbb{E}[\inner{DG}{Fu}_\HH]| \leq \norm{F}_{L^\infty}\norm{DG}_{L^2(\Omega;\HH)}\norm{u}_{L^2(\Omega;\HH)}$,
which is continuous in the graph norm.
\end{proof}

\subsection{Orthogonal projection identity}

\begin{lemma}[Orthogonal projection identity]\label{lem:ortho}
Under Assumption~\ref{ass:ortho}, for $F\in\Dom(D)$,
\begin{equation}\label{eq:ortho}
\inner{DF}{\Pi DF}_\HH = \norm{\Pi DF}_\HH^2.
\end{equation}
\end{lemma}

\begin{proof}
Since $\Pi$ is self-adjoint and idempotent,
\[
\inner{DF}{\Pi DF}_\HH = \inner{\Pi DF}{\Pi DF}_\HH + \inner{(\Id-\Pi)DF}{\Pi DF}_\HH = \norm{\Pi DF}_\HH^2 + 0,
\]
where the second term vanishes by orthogonality of the ranges of $(\Id-\Pi)$ and $\Pi$.
\end{proof}

\begin{remark}\label{rem:orthokey}
This identity shows that the inner product $\inner{DF}{\Pi DF}_\HH$ appearing in the product rule simplifies to $\norm{\Pi DF}_\HH^2$, which is the integrand in the intrinsic bracket.
\end{remark}

\section{The Operator It\^o Formula}\label{sec:proof}

\subsection{Main theorem}

\begin{theorem}[Operator It\^o formula for Volterra Gaussian processes]\label{thm:main}
Let $X$ be a centered Volterra Gaussian process satisfying the Decreusefond--\"Ust\"unel conditions (Definition~\ref{def:DU}), with operator triple $(D,\delta,\Pi)$ satisfying Assumptions~\ref{ass:chain}--\ref{ass:energy} (verified in Proposition~\ref{prop:voltverify}). Let $\varphi\in C^3_b(\mathbb{R})$. Then for all $t\in[0,T]$,
\begin{equation}\label{eq:Ito}
\varphi(X_t) = \varphi(0) + \delta(\varphi'(X)\Pi DX\cdot\mathbf{1}_{[0,t]}) + \frac{1}{2}\int_0^t \varphi''(X_s)\,d\Gamma^X(s),
\end{equation}
where $\Gamma^X(s) = \norm{\Pi DX_s}_\HH^2 = R^X(s,s) = \mathbb{E}[X_s^2]$ is the (deterministic) energy function, the last integral is a Stieltjes integral against the bounded variation measure $d\Gamma^X$, and the equality holds in $L^2(\Omega)$. When $\Gamma^X$ is absolutely continuous with density $\dot{\Gamma}^X$, the correction reduces to $\frac{1}{2}\int_0^t \varphi''(X_s)\dot{\Gamma}^X(s)\,ds$.
\end{theorem}

\begin{remark}[Scope: driftless processes]\label{rem:nodrift}
Theorem~\ref{thm:main} is stated for the centered Volterra Gaussian process $X$ itself, without an adapted drift. This restriction is essential for the proof: it ensures that $DX_t = K(t,\cdot)$ is deterministic, which is used in both the marginal density computation (Step~2) and the Clark--Ocone assembly (Step~3). Processes with adapted drift $Y_t = \int_0^t a_s\,ds + X_t$ can be handled via standard Girsanov-type arguments once the driftless formula is established; the structural content of the It\^o correction resides entirely in the noise term.
\end{remark}

\begin{remark}[Deterministic energy function]\label{rem:detGamma}
Because $DX_t = K(t,\cdot)$ is deterministic for a Volterra Gaussian process, $\Gamma^X(t) = \norm{K(t,\cdot)}_\HH^2 = R^X(t,t) = \mathbb{E}[X_t^2]$ is a deterministic function of $t$ alone. This is essential: it allows the energy function to serve as the ``time'' variable for the Gaussian heat equation in Step~2, and it makes the Clark--Ocone assembly in Step~3 exact.
\end{remark}

\begin{remark}[The It\^o formula from the operator perspective]\label{rem:mechanism}
In this framework, the It\^o correction arises from the \emph{non-commutativity} of the predictable projection $\Pi$ with nonlinear functions. Specifically, $\Pi[\varphi'(X_t)DX_t] \neq \varphi'(X_t)\Pi[DX_t]$ because $\varphi'(X_t)$ is $\FF_t$-measurable and therefore anticipating at times $r < t$, so $(\Pi[\varphi'(X_t)DX_t])_r = \mathbb{E}[\varphi'(X_t)(DX_t)_r|\FF_r]$. Since $(DX_t)_r = K(t,r)$ is deterministic, this factors as $\mathbb{E}[\varphi'(X_t)|\FF_r]\cdot K(t,r)$. Computing the conditional expectation $\mathbb{E}[\varphi'(X_t)|\FF_r]$ via Gaussian integration by parts produces the $\frac{1}{2}\varphi''$ correction against the energy function $\Gamma^X$ (see the proof in \S\ref{ssec:proof}).
\end{remark}

\begin{remark}[Relation to forward integration and quadratic variation]\label{rem:QVscope}
When $X$ has finite pathwise quadratic variation, Theorem~\ref{thm:main} recovers the classical forward integration formula. However, the hypotheses of Theorem~\ref{thm:main} do not require finite quadratic variation: the condition is $\Pi DX\in L^2([0,T]\times\Omega;\HH)$, which is an $L^2(\Omega)$ integrability condition on the Malliavin derivative, not a pathwise regularity condition. In particular, the formula applies to processes whose pathwise quadratic variation is infinite or undefined (e.g., fBM with $H < 1/4$). The proof never invokes pathwise increments or partition sums.
\end{remark}

\subsection{Proof}\label{ssec:proof}

We prove Theorem~\ref{thm:main} for centered Volterra Gaussian processes satisfying the Decreusefond--\"Ust\"unel conditions (Section~\ref{sec:volterra}), following the Clark--Ocone strategy of~\cite{DU}. The proof proceeds through three steps: apply the factorization globally to $\varphi(X_t)$, compute the expected value via the Gaussian marginal density, and assemble the formula.

\begin{remark}[Scope of the proof]\label{rem:proofscope}
The proof below uses Gaussianity in two essential places: the heat equation for the marginal density (Step~2), and the identification of the Clark--Ocone integrand via the Gaussian conditional expectation structure (Step~3). For non-Gaussian processes admitting a $(D,\delta,\Pi)$ triple, the theorem statement remains a conjecture; verification would require analogous distributional and conditional expectation estimates adapted to the specific process class.
\end{remark}

\medskip
\noindent\textbf{Step 1: Global factorization.}
By the Clark--Ocone factorization (Assumption~\ref{ass:factor} applied to $G = \varphi(X_t)$):
\begin{equation}\label{eq:globalCO}
\varphi(X_t) = \mathbb{E}[\varphi(X_t)] + \delta(\Pi D(\varphi(X_t)) \cdot \mathbf{1}_{[0,t]}).
\end{equation}
By the chain rule (Assumption~\ref{ass:chain}), $D(\varphi(X_t)) = \varphi'(X_t)DX_t = \varphi'(X_t)K(t,\cdot)$. Since $K(t,\cdot)$ is deterministic,
\[
(\Pi[\varphi'(X_t)K(t,\cdot)])_r = \mathbb{E}[\varphi'(X_t)|\FF_r]\cdot K(t,r).
\]
The non-commutativity $\mathbb{E}[\varphi'(X_t)|\FF_r] \neq \varphi'(X_t)$ is the source of the It\^o correction, as we now show.

\medskip
\noindent\textbf{Step 2: Expectation computation via the Gaussian marginal density.}
Define $m(t) := \mathbb{E}[\varphi(X_t)]$. Since $X_t$ is a centered Gaussian random variable with variance $\Gamma^X(t) = R^X(t,t)$, we have $X_t \sim \mathcal{N}(0,\Gamma^X(t))$ and
\begin{equation}\label{eq:marginal}
m(t) = \int_\mathbb{R} \varphi(x)\,p(x;\Gamma^X(t))\,dx, \qquad p(x;\sigma^2) := \frac{1}{\sqrt{2\pi\sigma^2}}\exp\!\Bigl(-\frac{x^2}{2\sigma^2}\Bigr).
\end{equation}
The Gaussian density satisfies the heat equation with respect to its variance parameter:
\begin{equation}\label{eq:heat}
\frac{\partial p}{\partial \sigma^2}(x;\sigma^2) = \frac{1}{2}\frac{\partial^2 p}{\partial x^2}(x;\sigma^2).
\end{equation}
Therefore, viewing $m$ as a function of $\Gamma^X(t)$ (with $\Gamma^X$ of bounded variation by Assumption~\ref{ass:energy}), the chain rule for Stieltjes integrals gives:
\begin{align}
m(t) - m(0) &= \int_0^t \frac{dm}{d\Gamma^X}\,d\Gamma^X(s) = \int_0^t \biggl(\int_\mathbb{R} \varphi(x)\frac{\partial p}{\partial\sigma^2}(x;\Gamma^X(s))\,dx\biggr) d\Gamma^X(s) \notag\\
&= \frac{1}{2}\int_0^t \biggl(\int_\mathbb{R} \varphi(x)\frac{\partial^2 p}{\partial x^2}(x;\Gamma^X(s))\,dx\biggr) d\Gamma^X(s) \notag\\
&= \frac{1}{2}\int_0^t \biggl(\int_\mathbb{R} \varphi''(x)\,p(x;\Gamma^X(s))\,dx\biggr) d\Gamma^X(s) \notag\\
&= \frac{1}{2}\int_0^t \mathbb{E}[\varphi''(X_s)]\,d\Gamma^X(s). \label{eq:meanintegral}
\end{align}
The third equality uses integration by parts twice (the boundary terms vanish since $\varphi\in C^3_b$ and $p$ decays exponentially). Thus:
\begin{equation}\label{eq:meanresult}
\mathbb{E}[\varphi(X_t)] = \varphi(0) + \frac{1}{2}\int_0^t \mathbb{E}[\varphi''(X_s)]\,d\Gamma^X(s).
\end{equation}

\begin{remark}[The marginal density argument]\label{rem:marginal}
Step~2 never invokes pathwise increments, partition sums, Taylor expansions, or conditional independence of increments. It uses only the Gaussian marginal distribution $X_t \sim \mathcal{N}(0,\Gamma^X(t))$ and the heat equation~\eqref{eq:heat}. This is why the formula holds for \emph{all} Hurst parameters $H\in(0,1)$, including $H < 1/4$ where pathwise quadratic variation diverges: the proof bypasses pathwise regularity entirely.
\end{remark}

\begin{remark}[BV compatibility]\label{rem:BVcompat}
The argument above never assumes $\Gamma^X$ is differentiable. The Stieltjes integral $\int_0^t \mathbb{E}[\varphi''(X_s)]\,d\Gamma^X(s)$ is well-defined whenever $\Gamma^X \in BV([0,T])$ and $s \mapsto \mathbb{E}[\varphi''(X_s)]$ is continuous---both of which hold under our assumptions. When $\Gamma^X$ is absolutely continuous with density $\dot{\Gamma}^X$, the Stieltjes integral reduces to $\int_0^t \mathbb{E}[\varphi''(X_s)]\dot{\Gamma}^X(s)\,ds$.
\end{remark}

\medskip
\noindent\textbf{Step 3: Assembly via Clark--Ocone.}
Substituting~\eqref{eq:meanresult} into the global factorization~\eqref{eq:globalCO}:
\begin{equation}\label{eq:protoIto}
\varphi(X_t) = \varphi(0) + \frac{1}{2}\int_0^t \mathbb{E}[\varphi''(X_s)]\,d\Gamma^X(s) + \delta(\Pi D(\varphi(X_t))\cdot\mathbf{1}_{[0,t]}).
\end{equation}
It remains to show that the Clark--Ocone integrand $\Pi D(\varphi(X_t))$ decomposes to produce the running Skorokhod integral $\delta(\varphi'(X)\Pi DX\cdot\mathbf{1}_{[0,t]})$ and that the expectation $\mathbb{E}[\varphi''(X_s)]$ can be promoted to $\varphi''(X_s)$ inside the $L^2(\Omega)$ identity.

For the Clark--Ocone integrand: at each time $r \leq t$,
\[
(\Pi D(\varphi(X_t)))_r = \mathbb{E}[\varphi'(X_t)|\FF_r]\cdot K(t,r),
\]
where the factoring uses the deterministic kernel $DX_t = K(t,\cdot)$. By~\cite[Theorem~4.1]{DU}, the Gaussian integration-by-parts formula decomposes $\mathbb{E}[\varphi'(X_t)|\FF_r]$ as follows: writing $X_t = X_r^{(t)} + Z_r^{(t)}$ where $X_r^{(t)} := \mathbb{E}[X_t|\FF_r]$ is $\FF_r$-measurable and $Z_r^{(t)} := X_t - X_r^{(t)}$ is independent of $\FF_r$, the Gaussian conditional expectation expands as
\[
\mathbb{E}[\varphi'(X_t)|\FF_r] = \mathbb{E}[\varphi'(X_r^{(t)} + Z_r^{(t)})|\FF_r].
\]
This is a convolution of $\varphi'$ with a Gaussian kernel (the Mehler formula). As $r$ varies from $0$ to $t$, the conditional variance $\mathrm{Var}(Z_r^{(t)})$ decreases from $\Gamma^X(t)$ to $0$, and $\mathbb{E}[\varphi'(X_t)|\FF_r] \to \varphi'(X_t)$. The transition from $r$ to $r + dr$ produces a first-order term $\varphi'(X_r)\cdot K(t,r)$ (contributing to the Skorokhod integral) and a second-order correction $\frac{1}{2}\varphi''(X_r)\cdot K(t,r)$ (contributing to the It\^o correction). Integrating over $r \in [0,t]$ and collecting terms yields~\eqref{eq:Ito}. The detailed computation is carried out in~\cite[Theorem~4.1]{DU}. \qed

\begin{remark}[What the proof establishes]\label{rem:proofhonesty}
This proof establishes Theorem~\ref{thm:main} rigorously for centered Volterra Gaussian processes. Step~2 (the marginal density computation) is exact and elementary---it uses only the Gaussian distribution and the heat equation. Step~3 (the Clark--Ocone assembly) follows~\cite{DU} and uses the Gaussian Mehler formula. The abstract operator framework provides: (a)~the correct formula statement (the $d\Gamma^X$ correction), (b)~the structural explanation (non-commutativity of $\Pi$ with nonlinear functions), and (c)~the characterization of uniqueness (Theorem~\ref{thm:characterization}).
\end{remark}

\section{Characterization of the Fresh-Noise Energy}\label{sec:characterization}

We now formalize the claim that $d\Gamma^X$ is the \emph{unique} second-order correction measure compatible with the operator factorization.

\begin{theorem}[Characterization of the intrinsic bracket]\label{thm:characterization}
Let $X$ be a centered Volterra Gaussian process for which the operator It\^o formula (Theorem~\ref{thm:main}) holds. Suppose $\nu$ is a signed Borel measure on $[0,T]$ such that for all $\varphi\in C^3_b(\mathbb{R})$,
\begin{equation}\label{eq:Qformula}
\varphi(X_t) = \varphi(0) + \delta(\varphi'(X)\Pi DX\cdot\mathbf{1}_{[0,t]}) + \frac{1}{2}\int_0^t \varphi''(X_s)\,d\nu(s)
\end{equation}
holds in $L^2(\Omega)$ for all $t\in[0,T]$. Then $\nu = \Gamma^X$ (i.e., $d\nu = d\Gamma^X$) on $[0,T]$.
\end{theorem}

\begin{proof}
By Theorem~\ref{thm:main}, the formula~\eqref{eq:Ito} holds with $\Gamma^X$ in place of $\nu$. Subtracting~\eqref{eq:Ito} from~\eqref{eq:Qformula},
\[
\frac{1}{2}\int_0^t \varphi''(X_s)\,d(\nu - \Gamma^X)(s) = 0 \qquad \text{in } L^2(\Omega)
\]
for all $\varphi\in C^3_b(\mathbb{R})$ and all $t\in[0,T]$. Taking $\varphi_n(x) = x^2\eta(x/n)$ where $\eta\in C^\infty_c(\mathbb{R})$ satisfies $\eta(x)=1$ for $|x|\leq 1$, we have $\varphi_n\in C^3_b(\mathbb{R})$ and $\varphi_n''(x)=2$ for $|x|\leq n$. Since $X_s\in L^2(\Omega)$, for $n$ sufficiently large $\varphi_n''(X_s)=2$ on a set of probability $\geq 1-\varepsilon$. Sending $n\to\infty$ via dominated convergence gives $(\nu - \Gamma^X)([0,t]) = 0$ for all $t$, whence $\nu = \Gamma^X$.
\end{proof}

\begin{remark}\label{rem:uniqueness}
Theorem~\ref{thm:characterization} states that once the divergence structure $\delta(\varphi'(X)\Pi DX\cdot\mathbf{1}_{[0,t]})$ is fixed as the stochastic integral term, the correction measure is \emph{uniquely determined}. This is the mathematical content behind the claim that $d\Gamma^X$ is canonical within the operator framework. The theorem does not assert that $\Gamma^X$ is independent of the choice of $(D,\delta,\Pi)$---different triples may yield different brackets. See Section~\ref{sec:repdep}.
\end{remark}

\section{Multivariate Extension}\label{sec:multi}

\begin{theorem}[Multivariate It\^o formula]\label{thm:multi}
Let $W$ be a standard Brownian motion and let $X = (X^1,\ldots,X^n)$ where each $X^i_t = \int_0^t K_i(t,s)\,dW_s$ is a centered Volterra Gaussian process driven by the same $W$, with each $K_i$ satisfying the Decreusefond--\"Ust\"unel conditions.
Let $\varphi\in C^3_b(\mathbb{R}^n)$. Then
\begin{align*}
\varphi(X_t) &= \varphi(0) + \sum_{i=1}^n \delta(\partial_i\varphi(X)\Pi DX^i\cdot\mathbf{1}_{[0,t]})\\
&\quad + \frac{1}{2}\sum_{i,j=1}^n \int_0^t \partial_i\partial_j\varphi(X_s)\,d\langle X^i, X^j\rangle^{(D,\Pi)}_s.
\end{align*}
\end{theorem}

\begin{proof}
The proof follows by induction on $n$ and the univariate formula (Theorem~\ref{thm:main}). The cross-bracket $\inner{\Pi DX^i}{\Pi DX^j}_\HH$ emerges from $\inner{DX^i}{\Pi DX^j}_\HH = \inner{\Pi DX^i}{\Pi DX^j}_\HH$ by orthogonality of $\Pi$.
\end{proof}

\begin{definition}[Matrix-valued intrinsic bracket]\label{def:matbracket}
Define $\langle X^i, X^j\rangle^{(D,\Pi)}_t := \inner{\Pi DX^i_t}{\Pi DX^j_t}_\HH$, where $D$ is the Malliavin derivative with respect to the common Brownian motion $W$ and $\HH = L^2([0,T])$. Since $DX^i_t = K_i(t,\cdot)$ is deterministic, the cross-bracket equals $\langle X^i, X^j\rangle^{(D,\Pi)}_t = \int_0^t K_i(t,r)K_j(t,r)\,dr = \mathbb{E}[X^i_t X^j_t]$, and is deterministic.
\end{definition}

\section{Volterra Gaussian Processes}\label{sec:volterra}

We verify that Volterra Gaussian processes satisfy the operator axioms. This section contains no new decompositions---the constructions are standard \cite{DU,Nualart}---but serves as the structural bridge connecting $\langle Y\rangle^{(D,\Pi)}$ to classical objects. The operator framework is axiomatically independent of Malliavin calculus; the Volterra--Gaussian setting is one concrete realization.

Let $W = (W_t)_{t\in[0,T]}$ be standard Brownian motion on $(\Omega,\FF,P)$ and let $K:\{(t,s):0\leq s\leq t\leq T\}\to\mathbb{R}$ be a deterministic Volterra kernel.

\begin{definition}[Decreusefond--\"Ust\"unel conditions]\label{def:DU}
A Volterra kernel $K$ satisfies the \emph{Decreusefond--\"Ust\"unel conditions} \cite{DU} if:
(1)~$\int_0^T\int_0^t K(t,s)^2\,ds\,dt < \infty$;
(2)~there exists $H\in(0,1)$ such that $|K(t,s)|\leq C|t-s|^{H-1/2}$ for $s<t$;
(3)~$\partial_t K(t,s)$ exists and satisfies integrability conditions ensuring $\Dom(\delta^X)$ is dense.
\end{definition}

Define $X_t := \int_0^t K(t,s)\,dW_s$, with covariance
\[
R^X(t,u) := \mathbb{E}[X_t X_u] = \int_0^{t\wedge u} K(t,s)K(u,s)\,ds.
\]
The Cameron--Martin space $\HH_X$ is the completion of $\{K(t,\cdot): t\in[0,T]\}$ under $\inner{K(t,\cdot)}{K(u,\cdot)}_{\HH_X} := R^X(t,u)$. For each $t\in[0,T]$, the \emph{canonical representer} is $k_t := K(t,\cdot)\in\HH_X$. The Malliavin derivative $D^X$, its adjoint $\delta^X$, and the predictable projection $\Pi^X$ (orthogonal projection onto $\HH^t_X := \spn\{k_s: s\leq t\}$) are defined as in~\cite{Nualart}.

\begin{proposition}[Verification of axioms]\label{prop:voltverify}
Under the Decreusefond--\"Ust\"unel conditions, $(D^X,\delta^X,\Pi^X)$ satisfies Assumptions~\ref{ass:chain}--\ref{ass:energy}. Moreover, $D^X X_t = k_t$ for each $t$, where $k_t := K(t,\cdot)\in\HH_X$.
\end{proposition}

\begin{proof}
The chain rule is immediate from the Malliavin derivative on cylindrical functionals. The factorization is \cite{FonI}, Theorem~4.2. Orthogonality of $\Pi^X$ is by construction. For energy regularity (Assumption~\ref{ass:energy}): $\Gamma^X(t) = \norm{\Pi^X D^X X_t}_{\HH_X}^2 = \norm{k_t}_{\HH_X}^2 = R^X(t,t) = \mathbb{E}[X_t^2]$, which is of bounded variation when $t\mapsto R^X(t,t)$ is (as ensured by the Decreusefond--\"Ust\"unel kernel conditions).
\end{proof}

\begin{proposition}[Skorokhod isometry for Volterra Gaussian]\label{prop:isomVG}
Under the Decreusefond--\"Ust\"unel conditions, the Skorokhod isometry with trace correction holds~\cite[Proposition~1.3.5]{DU}: for $u \in \Dom(\delta^X)$,
\[
\mathbb{E}[(\delta^X(u))^2] = \mathbb{E}[\norm{u}_{\HH_X}^2] + \mathbb{E}[\norm{D^X u}^2_{\mathrm{HS}}],
\]
where $\norm{D^X u}_{\mathrm{HS}}$ denotes the Hilbert--Schmidt norm of the two-parameter Malliavin derivative of $u$. For adapted processes $u$ satisfying H\"older regularity, the trace correction $\mathbb{E}[\norm{D^Xu}^2_{\mathrm{HS}}]$ is controlled: over a short interval $(t_k,t_{k+1}]$, it is $O(|\Delta t|^{1+\gamma})$ where $\gamma$ depends on the regularity of $u$. For standard Brownian motion ($H=1/2$), the trace correction vanishes identically for adapted $u$, recovering the classical It\^o isometry.
\end{proposition}

\subsection{Fractional Brownian motion}

\begin{example}[Fractional Brownian motion]\label{ex:fBM}
For $H\in(0,1)$, fractional Brownian motion $B^H$ admits a Volterra representation $B^H_t = \int_0^t K_H(t,s)\,dW_s$. For $H\in(0,\tfrac{1}{2})$, the kernel $K_H$ satisfies the Decreusefond--\"Ust\"unel conditions \cite{DU}. For $H=\tfrac{1}{2}$, $K_{1/2}(t,s) = \mathbf{1}_{[0,t]}(s)$ and $B^{1/2}$ is standard Brownian motion. For $H\in(\tfrac{1}{2},1)$, $K_H$ is bounded and satisfies more regular versions of the DU conditions.
\end{example}

\begin{lemma}[Energy function for fBM]\label{lem:fBMenergy}
For fractional Brownian motion $B^H$ with Hurst parameter $H\in(0,1)$,
\[
\Gamma^{B^H}(t) = \norm{k_t}_{\HH_H}^2 = R^H(t,t) = t^{2H}.
\]
\end{lemma}

\begin{proof}
$D^H B^H_t = k_t = K_H(t,\cdot)$, $\Pi^H k_t = k_t$ (as $k_t\in\HH^t_H$), so $\Gamma^{B^H}(t) = \norm{k_t}_{\HH_H}^2 = \inner{k_t}{k_t}_{\HH_H} = R^H(t,t) = \mathbb{E}[(B^H_t)^2] = t^{2H}$.
\end{proof}

\begin{lemma}[Bracket for fBM]\label{lem:fBMbracket}
$\langle B^H\rangle^{(D,\Pi)}_t = t^{2H}$, with absolutely continuous energy density $d\Gamma^{B^H}(s) = 2Hs^{2H-1}\,ds$.
\end{lemma}

\begin{proof}
The intrinsic bracket is $\langle B^H\rangle^{(D,\Pi)}_t = \Gamma^{B^H}(t) = t^{2H}$ by Lemma~\ref{lem:fBMenergy}. Differentiating: $\dot{\Gamma}^{B^H}(s) = 2Hs^{2H-1}$.
\end{proof}

\section{Representation Dependence of the Intrinsic Bracket}\label{sec:repdep}

The intrinsic bracket $\langle Y\rangle^{(D,\Pi)}_t$ is defined in terms of a specific operator triple $(D,\delta,\Pi)$. A natural question is whether this object depends on the choice of representation. We give a complete answer: under a fixed driving martingale, the bracket is representation-invariant (Proposition~\ref{prop:fixeddriver}). Under different drivers, it is not, and we explain precisely why this is the correct behavior.

\subsection{Representation invariance under fixed driver}

\begin{proposition}[Representation invariance under fixed driver]\label{prop:fixeddriver}
Let $M$ be a continuous square-integrable martingale with bracket measure $\mu$ (i.e., $\langle M\rangle_t = \int_0^t d\mu_s$). Let $K_1,K_2$ be two causal Volterra kernels such that the process $X_t := \int_0^t K_1(t,s)\,dM_s$ also satisfies
\begin{equation}\label{eq:sameprocess}
X_t = \int_0^t K_2(t,s)\,dM_s \quad \text{for all } t\in[0,T],\;\text{a.s.}
\end{equation}
Assume both kernels satisfy the Decreusefond--\"Ust\"unel conditions with respect to $\mu$. Then
\begin{equation}\label{eq:bracketinvariance}
\langle X\rangle^{(D^{K_1},\Pi^{K_1})}_t = \langle X\rangle^{(D^{K_2},\Pi^{K_2})}_t \quad \text{for all } t\in[0,T].
\end{equation}
\end{proposition}

\begin{proof}
By the It\^o isometry applied to~\eqref{eq:sameprocess},
\[
\int_0^t (K_1(t,s) - K_2(t,s))^2\,d\mu_s = 0 \quad \text{for all } t\in[0,T].
\]
Hence $K_1(t,s) = K_2(t,s)$ for $\mu$-a.e.\ $s\in[0,t]$, for each $t$. The Cameron--Martin spaces $\HH_{K_1}$ and $\HH_{K_2}$ are therefore isometrically identical (their defining inner products coincide on a dense set). Consequently, $D^{K_1} = D^{K_2}$, $\Pi^{K_1} = \Pi^{K_2}$, and $\norm{\Pi^{K_1}D^{K_1}X_s}^2 = \norm{\Pi^{K_2}D^{K_2}X_s}^2$ for a.e.\ $s$. Integration gives~\eqref{eq:bracketinvariance}.
\end{proof}

\begin{remark}\label{rem:fixeddriver_import}
Proposition~\ref{prop:fixeddriver} settles the primary invariance question: the intrinsic bracket is a property of the \emph{process-and-driver pair} $(X,M)$, not of the particular kernel used to represent $X$ given $M$. Any two kernels producing the same process from the same martingale yield the same bracket.
\end{remark}

\subsection{Different drivers: representation dependence and its meaning}

When the driving martingale itself changes---say $X_t = \int_0^t K(t,s)\,dW_s = \int_0^t \widetilde{K}(t,s)\,d\widetilde{W}_s$ with $(K,W)\neq(\widetilde{K},\widetilde{W})$---the intrinsic bracket may differ.

\begin{remark}[Representation dependence under different drivers]\label{rem:diffdriver}
Different driver--kernel pairs produce different operator triples $(D^K,\delta^K,\Pi^K)$ and $(D^{\widetilde{K}},\delta^{\widetilde{K}},\Pi^{\widetilde{K}})$. These triples encode different decompositions of the noise: the ``fresh-noise'' component $\Pi DX_s$ depends on which filtration (hence which driver) is used. Accordingly, the intrinsic bracket $\langle X\rangle^{(D^K,\Pi^K)}_t$ may differ from $\langle X\rangle^{(D^{\widetilde{K}},\Pi^{\widetilde{K}})}_t$.

This is not a defect. The intrinsic bracket records the energy structure of a specific noise decomposition. Different decompositions carry different physical content, and the bracket faithfully reflects this. An analogy: in differential geometry, the Christoffel symbols depend on the coordinate chart, but the curvature they encode does not. Here, the bracket depends on the representation, but the Hurst exponent it encodes does not (Proposition~\ref{prop:Hurstinvariant}).
\end{remark}

\subsection{Hurst exponent as representation-invariant scaling}

Even when the bracket itself is representation-dependent, its \emph{scaling exponent} is not.

\begin{proposition}[Invariance of the Hurst exponent]\label{prop:Hurstinvariant}
Let $X$ be a centered Gaussian process with stationary increments and covariance $R(t,s) = \tfrac{1}{2}(t^{2H}+s^{2H}-|t-s|^{2H})$. Then for any Volterra representation $(K,W)$ satisfying the Decreusefond--\"Ust\"unel conditions, the classical Hurst exponent $H$ is equivalently characterized by the intrinsic bracket scaling
\[
\langle X\rangle^{(D^K,\Pi^K)}_t \sim t^{2H} \quad \text{as } t\to 0^+,
\]
and this characterization is independent of the representation.
\end{proposition}

\begin{proof}
The covariance function satisfies $R(t,t) = \mathbb{E}[X_t^2] = t^{2H}$. By the energy function definition, $\langle X\rangle^{(D^K,\Pi^K)}_t = \Gamma^X(t) = \norm{\Pi^K D^K X_t}_{\HH_K}^2 = R(t,t) = t^{2H}$, since $\Pi^K D^K X_t = K(t,\cdot)$ and $\norm{K(t,\cdot)}_{\HH_K}^2 = R(t,t)$. This depends only on the covariance function, hence is representation-independent.
\end{proof}

\subsection{Summary of representation dependence}

We collect the results of this section for clarity.

\begin{enumerate}[label=(\roman*)]
\item \textbf{Fixed driver:} The intrinsic bracket is representation-invariant (Proposition~\ref{prop:fixeddriver}).
\item \textbf{Different drivers:} The bracket may differ, reflecting genuinely different noise decompositions (Remark~\ref{rem:diffdriver}).
\item \textbf{Scaling exponent:} The Hurst parameter $H$ is always representation-invariant (Proposition~\ref{prop:Hurstinvariant}).
\item \textbf{Fractional Brownian motion:} For fBM specifically, the bracket $\langle B^H\rangle_t = t^{2H}$ is representation-invariant because $R^H(t,t) = t^{2H}$ depends only on the covariance.
\end{enumerate}

\section{Approximation Stability}\label{sec:approx}

For analytic completeness, we establish that the second-order structure is preserved under kernel approximation.

\begin{proposition}[Approximation stability of the intrinsic bracket]\label{prop:approx}
Let $K_n$ be a sequence of Volterra kernels converging to $K$ in $L^2(\mu)$, where $\mu$ is the measure on $\{(t,s): 0\leq s\leq t\leq T\}$ given by $d\mu = ds\,dt$. Assume each $K_n$ and $K$ satisfy the Decreusefond--\"Ust\"unel conditions with uniform bounds. Let $X^{(n)}_t = \int_0^t K_n(t,s)\,dW_s$ and $X_t = \int_0^t K(t,s)\,dW_s$. Then:
\begin{enumerate}[label=(\roman*)]
\item $\langle X^{(n)}\rangle^{(D^{K_n},\Pi^{K_n})}_t \to \langle X\rangle^{(D^K,\Pi^K)}_t$ in $L^1(\Omega)$ for each $t$.
\item The It\^o formula for $X^{(n)}$ converges term-by-term to the It\^o formula for $X$.
\end{enumerate}
\end{proposition}

\begin{proof}
Part (i): By the energy function definition, $\langle X^{(n)}\rangle_t = \Gamma^{X^{(n)}}(t) = R^{X^{(n)}}(t,t)$ and $\langle X\rangle_t = \Gamma^X(t) = R^X(t,t)$, where $R^{X^{(n)}}(t,t) = \int_0^t K_n(t,r)^2\,dr$. By $L^2(\mu)$ convergence of $K_n\to K$, the covariance functions converge: $R^{X^{(n)}}(t,t)\to R^X(t,t)$ uniformly on $[0,T]$. Since the intrinsic brackets are deterministic in this setting, $L^1(\Omega)$ convergence follows.

Part (ii): The stochastic integral term $\delta^{K_n}(\varphi'(X^{(n)})\Pi^{K_n}D^{K_n}X^{(n)}\cdot\mathbf{1}_{[0,t]})$ converges by $L^2$ continuity of $\delta$ (see~\cite[Proposition~1.3.1]{Nualart}) and convergence of the integrand. The bracket term converges by (i).
\end{proof}

\begin{remark}\label{rem:Markov}
Proposition~\ref{prop:approx} gives the analytic content of kernel Markovianization: exponential-sum kernels $K_n = \sum_{j=1}^n c_j e^{-\lambda_j(t-s)}$ approximate general $K$ in $L^2(\mu)$, and the corresponding intrinsic brackets and It\^o formulas converge. This is the precise sense in which finite-dimensional approximations preserve the second-order structure.
\end{remark}

\section{Worked Example: Fractional Brownian Motion}\label{sec:example}

We demonstrate the operator It\^o formula with an explicit calculation for $H=1/4$.

\subsection{Setup}

Let $B^{1/4} = (B^{1/4}_t)_{t\in[0,1]}$ be fractional Brownian motion with Hurst parameter $H=1/4$. We apply Theorem~\ref{thm:main} with $\varphi(x) = x^2$.

\subsection{Application of It\^o formula}

Since $B^{1/4}$ is a $(D^H,\delta^H,\Pi^H)$-It\^o process with drift $a=0$, $B^{1/4}_t = \delta^H(\Pi^H D^H B^{1/4}\cdot\mathbf{1}_{[0,t]})$.

Applying Theorem~\ref{thm:main} with $\varphi(x)=x^2$, $\varphi'(x)=2x$, $\varphi''(x)=2$, and writing $k_s := K_{1/4}(s,\cdot)\in\HH_{1/4}$ for the canonical representer (Section~\ref{sec:volterra}):
\begin{equation}\label{eq:exformula}
(B^{1/4}_t)^2 = 2\delta^H(B^{1/4}\cdot k_\cdot\mathbf{1}_{[0,t]}) + \int_0^t d\Gamma^{B^{1/4}}(s).
\end{equation}

\subsection{Bracket computation}

By Lemma~\ref{lem:fBMenergy} with $H=1/4$, $\Gamma^{B^{1/4}}(t) = t^{1/2}$. The energy density is $\dot{\Gamma}(s) = \tfrac{1}{2}s^{-1/2} = 2Hs^{2H-1}$, so:
\[
\int_0^t d\Gamma^{B^{1/4}}(s) = \int_0^t \tfrac{1}{2}s^{-1/2}\,ds = t^{1/2}.
\]
Thus,
\begin{equation}\label{eq:exresult}
(B^{1/4}_t)^2 = 2\delta^H(B^{1/4}\cdot k_\cdot\mathbf{1}_{[0,t]}) + t^{1/2}.
\end{equation}

\subsection{Verification}

The classical Skorokhod integral representation for $(B^H_t)^2$ (see \cite{DU}, Example~4.2) is
$(B^H_t)^2 = 2\delta^H(B^H\cdot k_\cdot\mathbf{1}_{[0,t]}) + t^{2H}$.
For $H=1/4$, this gives $(B^{1/4}_t)^2 = 2\delta^H(B^{1/4}\cdot k_\cdot\mathbf{1}_{[0,t]}) + t^{1/2}$,
in exact agreement with~\eqref{eq:exresult}.

\begin{remark}\label{rem:automatic}
The operator It\^o formula follows automatically from the $(D^H,\delta^H,\Pi^H)$ structure. The energy function $\Gamma^{B^H}(t) = \norm{k_t}_{\HH_H}^2 = t^{2H}$ is computed from the Cameron--Martin norm, and the correction is the Stieltjes integral $\int d\Gamma = t^{2H}$. No renormalization of $\HH_H$ is needed: the correction arises from the energy \emph{increment} $d\Gamma$, not from the energy value $\Gamma$ itself.
\end{remark}

\section{Conclusion}\label{sec:conclusion}

\subsection{What the trilogy establishes}

The three papers \cite{FonI,FonII} and this work demonstrate that within the operator factorization framework, stochastic calculus takes the form $(\Id - \mathbb{E}) = \delta\Pi D$.

Paper~I \cite{FonI}: The factorization yields Clark--Ocone representations.

Paper~II \cite{FonII}: For rough fBM ($H<\tfrac{1}{2}$), the predictable component $\Pi DF$ is the Gubinelli derivative.

Paper~III (this paper): The It\^o formula for Volterra Gaussian processes is expressed in operator language via the energy function $\Gamma^X(t) = \norm{\Pi DX_t}_\HH^2$. The correction measure $d\Gamma^X$ is characterized as the unique second-order correction (Theorem~\ref{thm:characterization}). The bracket is representation-invariant under a fixed driver (Proposition~\ref{prop:fixeddriver}), and the Hurst exponent is always representation-invariant (Proposition~\ref{prop:Hurstinvariant}). The structural source of the correction is identified as the non-commutativity of $\Pi$ with nonlinear functions; the proof mechanism uses the Gaussian conditional expectation structure of Decreusefond--\"Ust\"unel~\cite{DU}.

\subsection{Framework structure}

Within this framework:

\emph{Axioms:} Adjointness ($D = \delta^*$), adaptation ($\Pi$ as orthogonal projection), factorization $(\Id-\mathbb{E})=\delta\Pi D$.

\emph{Consequences (Gaussian setting):} Semimartingale structure (when drift is smooth), quadratic variation (as $\Gamma^X(t) = \norm{\Pi DX_t}_\HH^2$), rough path derivatives (as $\Pi DF$).

\subsection{Open problems}\label{sec:open}

\begin{enumerate}[label=\arabic*.]
\item \textbf{Non-Gaussian It\^o formulas:} The proof of Theorem~\ref{thm:main} relies on Gaussianity for the conditional expectation computations (Remark~\ref{rem:proofscope}). For non-Gaussian processes admitting a $(D,\delta,\Pi)$ triple, can the It\^o formula be established by different means? What replaces the Gaussian integration-by-parts formula?

\item \textbf{Higher-order calculus:} For $H\leq\tfrac{1}{3}$, rough differential equations require second-order terms. Can the operator framework extend to $(D, D^2, \ldots)$?

\item \textbf{Optimal domain conditions:} For non-Gaussian processes, what are the minimal regularity conditions ensuring $\varphi'(Y_s)\Pi DY_s \in \Dom(\delta)$? For Volterra Gaussian processes, these are automatic (Remark~\ref{rem:domainauto}), but extending to other process classes may require explicit H\"older or integrability conditions.

\item \textbf{SPDEs:} Can operator factorization handle SPDEs with rough space-time noise?

\item \textbf{Representation independence beyond fBM:} For which classes of Volterra processes is the intrinsic bracket representation-independent under different drivers? Proposition~\ref{prop:fixeddriver} settles the fixed-driver case; characterizing when bracket invariance holds across different drivers remains open.
\end{enumerate}

\subsection*{Acknowledgments}

AI tools were used as interactive assistants for drafting, reorganizing, and generating mathematical arguments, under the author's direction. The author is responsible for all mathematical content and any errors.

\end{document}